\documentclass[11pt]{amsart}

\usepackage{amsfonts}
\usepackage{amsmath}
\usepackage[utf8]{inputenc}
\usepackage[english]{babel}
\usepackage{amsthm}
\usepackage{color}
\usepackage{hyperref}

\usepackage{pgfplots}

\usepackage{amssymb}
\usepackage{amsopn}
\usepackage{mathrsfs}

\usepackage{todonotes}
\usepackage[left=3cm,right=3cm,top=3cm,bottom=3cm]{geometry}

\DeclareMathOperator{\diam}{diam}		
\DeclareMathOperator{\supp}{supp}

\newcommand{\N}{\mathbb{N}}
\newcommand{\R}{\mathbb{R}}

\renewcommand{\S}{\mathbb{S}}
\renewcommand{\H}{\mathbb{H}}

\newcommand{\be}{\begin{equation}}
\newcommand{\ee}{\end{equation}}

\newcommand{\nH}{\nabla\!_\mathbb{H}}

\DeclareMathOperator{\idty}{Id}

\newcommand{\bH}{\mathbb{H}}
\newcommand{\bS}{\mathbb{S}}
\newcommand{\cL}{\mathcal{L}}
\newcommand{\cZ}{\mathcal{Z}}

\DeclareMathOperator{\bR}{\mathbb{R}}

\newcommand{\distr}{\mathcal{D}}			

\theoremstyle{plain}        
\newtheorem{thm}{Theorem}
\newtheorem{prop}[thm]{Proposition}    
\newtheorem{cor}[thm]{Corollary}      
\newtheorem{lem}[thm]{Lemma}         
\theoremstyle{definition}     
\newtheorem{defn}[thm]{Definition}
   
\theoremstyle{remark}        
\newtheorem{rem}[thm]{Remark}

\usepackage[normalem]{ulem}

\newcommand{\ndperp}{\Xi}
\newcommand{\nd}{\nH\delta}

\title[Hardy-type inequalities in $\bH^n$]{Hardy-type inequalities for the Carnot-Carathéodory distance in the Heisenberg group}
\author[V.~Franceschi]{V.~Franceschi$^\flat$}
\address{$^\flat$ FMJH \& IMO, Universit\'e Paris-Sud, 91405 Orsay, France}
\email{valentina.franceschi@math.u-psud.fr}
\author[D.~Prandi]{D.~Prandi$^\sharp$}
\address{$^\sharp$ CNRS, L2S, CentraleSup\'elec, France}
\email{dario.prandi@l2s.centralesupelec.fr}

\date{\today}

\begin{document}
\maketitle

\begin{abstract}
  In this paper we study various Hardy inequalities in the Heisenberg group $\H^n$, w.r.t.\ the Carnot-Carathéodory distance $\delta$ from the origin. We firstly show that the optimal constant for the Hardy inequality is strictly smaller than $n^2 = (Q-2)^2/4$, where $Q$ is the homogenous dimension. Then, we prove that, independently of $n$, the Heisenberg group does not support a radial Hardy inequality, i.e., a Hardy inequality where the gradient term is replaced by its projection along $\nH\delta$. This is in stark contrast with the Euclidean case, where the radial Hardy inequality is equivalent to the standard one, and has the same constant.
  
  Motivated by these results, we consider Hardy inequalities for non-radial directions, i.e., directions tangent to the Carnot-Carathéodory balls. In particular, we show that the associated constant is bounded on homogeneous cones $C_\Sigma$ with base $\Sigma\subset \S^{2n}$, even when $\Sigma$ degenerates to a point. This is a genuinely sub-Riemannian behavior, as such constant is well-known to explode for homogeneous cones in the Euclidean space.
\end{abstract}

\section{Introduction}


%
%
%
The classical Hardy inequality states that
\begin{equation}\label{eq:hardy-eucl}
 	\int_{\bR^d} |\nabla u|^2 \, dp \ge \left(\frac{d-2}{2}\right)^2 \int_{\bR^d} \frac{|u|^2}{|p|^2}\,dp, \qquad \forall u\in C^\infty_c(\bR^d\setminus\{0\}),
\end{equation}
where the constant on the r.h.s.\ is sharp.
Here, $|p|$ denotes the Euclidean distance from the origin, and $|\nabla u|^2$ is the squared norm of the Euclidean gradient.
In the Euclidean setting, the above is actually equivalent to the so-called ``radial'' Hardy inequality:
\begin{equation}\label{eq:hardy-rad-eucl}
 	\int_{\bR^d} \left|\nabla u(p)\cdot\frac{p}{|p|}\right|^2 \, dp \ge \left(\frac{d-2}{2}\right)^2 \int_{\bR^d} \frac{|u|^2}{|p|^2}\,dp, \qquad \forall u\in C^\infty_c(\bR^d\setminus\{0\}).
\end{equation}
This can  be proved via polar coordinates, and trivially implies \eqref{eq:hardy-eucl}. Then, one can show the sharpness of \eqref{eq:hardy-eucl} by explicitly finding a radial minimizing sequence.

In this paper we are interested in extensions of the above inequalities to the Heisenberg setting (see \cite{Agrachev-book,Capogna-book}). For $n\in \N$, the Heisenberg group $\H^n$ is $\bR^{2n+1}$ endowed with the sub-Riemannian structure generated by the $2n$-dimensional distribution  $\mathcal D \subset T\bR^{2n+1}$ with orthonormal frame:
\begin{equation}\label{eq:XY}
  X_i = \partial_{x_i} - \frac{y_i}{2}\partial_z,
  \qquad
  Y_i = \partial_{y_i} + \frac{x_i}{2}\partial_z, \qquad i =1,\ldots, n.
\end{equation} 
Here, we denoted points in $\bR^{2n+1}$ by $(x,y,z)\in\bR^n\times\bR^n\times\bR$. The resulting structure is step $2$, with the only non-trivial commutators being $[X_i,Y_i]=Z$, $i=1,\ldots, n$, where $Z=\partial_z$. 
We denote by $\delta$ the Carnot-Carathéodory distance from the origin, and for all $\lambda>0$ we let the anisotropic homogeneous dilations of $\bH^n$ to be
\begin{equation}\label{eq:dilation}
  \varrho_\lambda : \bR^{2n+1}\to \bR^{2n+1}, \qquad  \varrho_\lambda(x,y,z) = (\lambda x, \lambda y, \lambda^2 z), \qquad \lambda>0.
\end{equation}
Recall that $\delta$ is $1$-homogeneous w.r.t.\ $\varrho_\lambda$, i.e., $\delta(\varrho_\lambda p) = \lambda \delta(p)$, $\lambda>0$, $p\in \bH^n$. (See Section~\ref{sec:prelim} for more details.) 

The literature regarding Hardy inequalitiees in the Heisenberg setting is extensive.
Garofalo and Lanconelli \cite{Garofalo1990} proved the following weighted Hardy-type inequality
\begin{equation}\label{eq:garofalo}
  \int_{\bH^n}|\nH u|^2\,dp \ge n^2 \int_{\bH^n}\frac{|u|^2}{N^2}|\nH N|^2\,dp,\qquad
  u\in C^\infty_c(\bH^n\setminus\{0\}).
\end{equation}
Here, $\nH u$ denotes the horizontal gradient of $u$, i.e., the gradient along the directions \eqref{eq:XY}, while $N(x,y,z):=\sqrt[4]{(|x|^2+|y|^2)^2+16 z^2}$ is the Korányi gauge. This is an homogeneous norm equivalent to $\delta$, associated with the fundamental solution of the sub-Laplacian. Inequality \eqref{eq:garofalo} is sharp, and the constant is the correct equivalent of the Euclidean Hardy constant, obtained by replacing the Euclidean dimension by the homogeneous dimension $Q=2n+2$.
The main drawback of \eqref{eq:garofalo} is that the weight $|\nH N|$ is singular on the center $\cZ=\{x=y=0\}$. 
Unweighted (but not sharp) Hardy inequalities in the Heisenberg group w.r.t.~the the Korányi gauge have been studied in, e.g., \cite{Bahouri2006,Bahouri2005}.
We refer to \cite{D'Ambrosio2004,NZY2001} for an $L^p$ analog of inequality \eqref{eq:garofalo}. Extensions to more general Carnot groups can be found in, e.g.,  \cite{GoldsteinKombe2008,GSY2018,RS17b}. For sub-Riemannian structures which are not Carnot groups see \cite{D'Ambrosio2003,D'Ambrosio2005,DGP2011,Kogoj2016} and references therein.

Concerning the Carnot-Carathéodory distance, Lehrback \cite{Lehrback2017} has shown that there exists a constant $c>0$ such that
\begin{equation}\label{eq:lehrback}
  \int_{\bH^n}|\nH u|^2\,dp \ge c \int_{\bH^n}\frac{|u|^2}{\delta^2}\,dp,\qquad
  u\in C^\infty_c(\bH^n\setminus\{0\}).
\end{equation}
(The same results follows also from the results in \cite{Bahouri2006,Bahouri2005}, using the fact that the Korányi gauge is equivalent to $\delta$.)
Unfortunately the proofs are based on very general techniques, which have no hope to yield any information about the optimal constant, that we henceforth denote by 
\begin{equation}
  c_n:=\sup\{c>0 \text{ s.t.\ \eqref{eq:lehrback} holds} \} = \inf_{u\in C^\infty_c(\H^n\setminus\{0\})} \frac{\int_{\bH^n}|\nH u|^2\,dp}{\int_{\bH^n}\frac{|u|^2}{\delta^2}\,dp}.
\end{equation} 
Such an optimal constant is claimed in \cite{Yang2013} to coincide with the one of \eqref{eq:garofalo}, i.e., $c_n=n^2$. The argument is based on divergence-like results on metric balls which should yield a radial Hardy inequality for $c>0$, i.e., 
\begin{equation}\label{eq:radial-heis}
  \int_{\bH^n}|\langle\nH u, \nH\delta\rangle|^2\,dp \ge c \int_{\bH^n}\frac{|u|^2}{\delta^2}\,dp,\qquad
  u\in C^\infty_c(\bH^n\setminus\{0\}).
\end{equation}
The claimed result is that $c_n^{\text{rad}} := \sup\{c>0 \text{ s.t.\ \eqref{eq:radial-heis} holds} \} \ge n^2$, which implies $c_n\ge n^2$.

Our first result, contained in Section~\ref{sec:yang}, raises a crucial criticism against the above claim. 
\begin{thm}\label{thm:counter-yang}
	For any $n\ge 1$, we have that
	\begin{equation}
		c_n^{\text{rad}}=0
		\qquad\text{and}\qquad
		c_n < n^2.
	\end{equation}
\end{thm}

Since \eqref{eq:lehrback} has to hold with $c_n>0$, the above implies that, in the Heisenberg case, the full Hardy inequality is not equivalent to the radial one. This is in stark contrast w.r.t.\ the Euclidean case.

The flaw of the argument in \cite{Yang2013} is the implicit assumption that, as in the Euclidean setting, the horizontal unit vector field $\nH\delta$ coincides with the generator of the dilations $\frac{d}{d\lambda}\varrho_\lambda$. However, in the Heisenberg setting the latter is not horizontal.
We fix this problem in Section~\ref{sec:garofalo}, and show that this technique yields exactly the weighted Hardy inequality \eqref{eq:garofalo}.

We also mention that in \cite{Ruzhansky2017} a radial Hardy inequality with sharp constant $n^2$ is obtained for any homogenous norm $\|\cdot\|$ (and thus also for $\delta$), by replacing the integrated quantity on the l.h.s.\ of \eqref{eq:radial-heis} with $|\frac{d}{d\|p\|}u|^2$, which takes into account non-horizontal directions.

\subsection{Non-radial Hardy inequalities on homogeneous cones}

Motivated by Theorem~\ref{thm:counter-yang}, in the second part of the paper we derive some new Hardy inequalities by considering the components of the horizontal gradient $\nH u$ along vector fields orthogonal to $\nH\delta$. These techniques are particularly well-adapted to derive Hardy inequalities on homogenous cones, but not on the full space. Let us briefly discuss what happens in the Euclidean case.

Let  $\Sigma\subset\bS^{d-1}$ be a spherical cap. The associated Euclidean cone, obtained via Euclidean (homogeneous) dilations, is $C^{\text{eucl}}_\Sigma=\{ \lambda p \mid \lambda>0, \, p\in \Sigma \}\subset \bR^d$. Hardy inequalities for these sets w.r.t.\ the Euclidean distance from the origin have been deeply investigated in the literature, see e.g., \cite{Fall2012}. 
%
In particular, letting $\nabla^\perp$ denote the components of the gradient orthogonal to the radial vector field $p/|p|$, the arguments in \cite{Fall2012} allow to prove that the following Hardy inequality holds
\begin{equation}
 	\int_{C^{\text{eucl}}_\Sigma} |\nabla^\perp u|^2 \, dp \ge c \int_{C^{\text{eucl}}_\Sigma} \frac{|u|^2}{|p|^2}\,dp, \qquad \forall u\in C^\infty_c(C^{\text{eucl}}_\Sigma), 
\end{equation}
with optimal constant $c_{d}^{\perp,\text{eucl}}=\lambda_1(\Sigma)$, where $\lambda_1(\Sigma)$ is the first Dirichlet eigenvalue of the Laplace-Beltrami operator on $\Sigma\subset \bS^{d-1}$. 
A trivial consequence of this fact is that $c_d^{\perp,\text{eucl}}(\Sigma)\rightarrow +\infty$ as $\Sigma$ degenerates to a point, i.e., when the cone degenerates to a half-line.
In the present paper, we show that this is not the case for homogeneous cones in $\H^n$, that we now introduce.

The Heisenberg homogenous cone associated with a spherical cap $\Sigma\subset \bS^{2n}$ is $C_\Sigma=\{ \varrho_\lambda(p)\mid \lambda>0,\, p\in \Sigma \}$, where $\varrho_\lambda$ are the dilations defined in \eqref{eq:dilation}. Observe that $\partial C_\Sigma$ is an Euclidean paraboloid. 
Although these sets are smooth (which is not the case for the Euclidean cones) the origin is a characteristic point for the sub-Riemannian structure, i.e., $C_\Sigma$ is tangent to the distribution at the origin. 
In the following we focus on spherical caps $\Sigma$ entirely contained in the upper hemisphere $\bS^{2n}_+=\{|x|^2+|y|^2+z^2=1\mid z>0\}$. The associated homogeneous cones are uniquely identified by a parameter $\alpha_\Sigma>0$ such that $C_\Sigma = \{(x,y,z)\in \bR^{2n+1}\mid |x|^2+|y|^2<\alpha_\Sigma z\}$. 

Similarly to the Euclidean case, we let $\nH^\perp$ denote the components of the horizontal gradient orthogonal to $\nH\delta$. (Observe that, given $u\in C^\infty(\H^n)$, the vector field $\nH^\perp u$ only makes sense on  $\H^n\setminus\cZ$, since $\nH\delta$ is not defined on $\cZ$.) 
We consider the following Hardy inequality on $C_\Sigma$, for $c>0$:
\begin{equation}\label{eq:hardy-cone}
    \int_{C_\Sigma} |\nH^\perp u|^2\,dp \ge c\int_{C_\Sigma} \frac{|u|^2}{\delta^2}\,dp,
    \qquad
    \forall u\in C^\infty_c(C_\Sigma),
\end{equation}
and let the optimal constant in the above to be
\begin{equation}
  c_n^\perp(\Sigma):=\sup\{ c>0\text{ s.t.\ \eqref{eq:hardy-cone} holds} \} = \inf_{u\in C^\infty_c(C_\Sigma)}\frac{\int_{C_\Sigma} |\nH^\perp u|^2\,dp}{\int_{C_\Sigma}|u|^2\delta^{-2}\,dp}.
\end{equation}

We have the following.

\begin{thm}\label{thm:hardy-intro1}
  Let $\Sigma\subset \bS^{2n}_+$ be a spherical cap and  $\rho_\Sigma= \phi^{-1}(\alpha_\Sigma)\in[0,2\pi]$,  where $\phi:[0,2\pi]\to [0,+\infty]$ is the order-reversing diffeomorphism defined by
  \begin{equation}\label{eq:phi}
    \phi(r) = \frac{4(1-\cos r )}{r-\sin r}.
  \end{equation}
  Then, we have
  \begin{equation}\label{eq:hardy1}
    \frac{n^2\rho_\Sigma^2}{4}\le c^\perp_n(\Sigma) \le \pi^2n^2.
  \end{equation}
\end{thm}


\begin{rem}
  For the upper-half plane $\bH^n_+ = \{z>0\} = C_{\bS^{2n}_+}$ we have $\alpha_{\bS^{2n}_+}=+\infty$ and $\rho_{\bS^{2n}_+} = 0$. In this case, the l.h.s.\ of \eqref{eq:hardy1} is trivial. On the other hand, for the center $\cZ = \{x=y=0\}=C_{\{(0,0,1)\}}$ we have  $\alpha_{\{(0,0,1)\}}=0$ and $\rho_{\{(0,0,1)\}} = 2\pi$. For Hardy inequalities in the half space we refer to \cite{LiuLuan13,LuanYang2008} and in more general convex domains we refer to \cite{Larson16,Ruszkowski18}. 
\end{rem}

\begin{rem}
  Theorem~\ref{thm:hardy-intro1} shows that, contrarily to the Euclidean case, $c_n^\perp(\Sigma)$ is always bounded. Nevertheless, in Section~\ref{sec:santalo} we prove that the full Hardy constant on $C_\Sigma$, defined by
  \begin{equation}\label{eq:full-hardy}
  c_n(\Sigma):= \inf_{u\in C^\infty_c(C_\Sigma)}\frac{\int_{C_\Sigma} |\nH u|^2\,dp}{\int_{C_\Sigma}|u|^2\delta^{-2}\,dp},
\end{equation}
  explodes to $+\infty$ as $\Sigma$ degenerates to a point. 
  Heuristically, this can be interpreted as the fact that the explosion of $c_n(\Sigma)$ is not provided by  the non-radial components of the gradient. This contrasts with the Euclidean setting, where the full Hardy constant is $c_d^{\text{eucl}}(\Sigma) = {(d-2)^2}/4+ c_d^{\perp,\text{eucl}}(\Sigma)$. See, e.g., \cite{Fall2012}.
\end{rem}

Although the bounds given in  Theorem~\ref{thm:hardy-intro1}  are probably not sharp, they are deduced from a weighted directional Hardy inequality which is sharp. In order to present it, we need to introduce some notation.


As detailed in Section~\ref{sec:prelim}, exploiting the explicit optimal synthesis of the Heisenberg's geodesics issued from the origin, one can introduce ``polar''-like coordinates $(t, \varpi, r)\in U = [0,+\infty)\times \bS^{2n-1}\times (-2\pi,2\pi)$ which parametrize $\bH^n\setminus\cZ$. 
In particular, the coordinates $t, r$ of a point $(\xi, z)\in \bH^n$ are defined by $t=\delta(\xi, z)$, and $r = \psi(\xi, z)$, where $\psi(\xi,z)= \phi^{-1}(|\xi|^2/z)$ for $\phi$ defined in \eqref{eq:phi}. Observe that $\psi(0,z) = \operatorname{sgn}(z)2\pi$ and $\psi(\xi,0)=0$.

Our main result is then the following.

\begin{thm}\label{thm:hardy-intro2}
  Let $\Sigma\subset \bS^{2n}_+$ be a spherical cap. Then, we have
  \begin{equation}\label{eq:hardy2}
    \int_{C_\Sigma} \frac{|\nH^\perp u|^2}{\psi}\,dp \ge \frac{n^2}{4}\int_{C_\Sigma} \frac{|u|^2}{\delta^2} \psi\,dp,
    \qquad
    \forall u\in C^\infty_c(C_\Sigma).
  \end{equation}
  Moreover, the above inequality is sharp, in the sense that
  \begin{equation}
    c_n^\perp(\Sigma,\psi) := \inf_{u\in C^\infty_c(C_\Sigma)}\frac{\int_{C_\Sigma} \psi^{-1}|\nH^\perp u|^2\,dp}{\int_{C_\Sigma} \psi{|u|^2}{\delta^{-2}}\,dp}  = \frac{n^2}{4}.
  \end{equation}
\end{thm}

\begin{rem}
As shown in Appendix~\ref{a:euclidean}, the arguments used in the above are valid also in the Euclidean case. In particular, in this case they allow to correctly predict the explosion of the constant $c^{\perp,\text{eucl}}_n(\Sigma)$ as $\Sigma$ degenerates to a point.  
\end{rem}

\subsection{Structure of the paper}

In Section~\ref{sec:prelim} we present the Heisenberg group and some preliminary constructions needed in the following, including suitable polar-like coordinates. We exploit the latter in Section~\ref{sec:yang} to prove Theorem~\ref{thm:counter-yang}, and in Section~\ref{sec:hardy-vert} to obtain Theorems~\ref{thm:hardy-intro1} and  \ref{thm:hardy-intro2}. Section~\ref{sec:santalo} is devoted to showing the explosion of the full Hardy constant $c_n(\Sigma)$ as $\Sigma$ degenerates to a point, via the techniques developed in \cite{prandi2019}. A fix for the arguments of \cite{Yang2013} proposed in Section~\ref{sec:garofalo}. Finally, in the appendix we discuss the Euclidean counterpart of Theorems~\ref{thm:hardy-intro1} and  \ref{thm:hardy-intro2}.

\subsection*{Acknowledgments} We thank N.\ Garofalo and L.\ Rizzi for fruitful discussions on the topic. We also thank Q.\ Yang for his kind replies to our questions. The authors acknowledge the support of ANR-15-CE40-0018 project \textit{SRGI - Sub-Riemannian Geometry and Interactions} and of a public grant overseen by the French National Research Agency (ANR) as part of the Investissement d'avenir program, through the iCODE project funded by the IDEX Paris-Saclay, ANR-11-IDEX-0003-02. The first author is supported by a public grant as part of the FMJH.

\section{Preliminaries}\label{sec:prelim}

On $\bR^{2n+1}$ we denote coordinates by $(\xi,z)\in \bR^{2n}\times\bR$, $\xi=(\xi_1,\ldots,\xi_n)$, $\xi_i=(x_i,y_i)$. The $n$-th Heisenberg group $\H^n$ is $\R^{2n+1}$ endowed with the sub-Riemannian structure with orthonormal frame \eqref{eq:XY}, satisfying the commutation relations $[X_i,Y_i]=\partial_z$, and $[X_i,Y_j]=0$ if $i\neq j$. 
These are left-invariant vector fields w.r.t.\ to the non-commutative group law
\begin{equation}\label{eq:group}
  (\xi,z)\star (\xi',z') = \left(\xi+\xi', z+z'+\frac{1}{2} \langle \xi,\tilde J\xi'\rangle_{\bR^{2n}}\right),
  \quad     
  \text{where}
  \quad
  J= \left(
    \begin{array}{cc}
      0 & 1\\
      -1 & 0
    \end{array}
  \right),
\end{equation}
and, for any matrix $M\in \cL(\bR^2)$, we let $\tilde M\in \cL(\bR^{2n})$ be the corresponding block diagonal operator, i.e., $\tilde M=\bigoplus_{j=1}^n M$.
The center of the group is $\cZ=\{\xi=0\}$.
Moreover, the sub-Riemannian distribution $\distr = \operatorname{span}\{X_1,\ldots,X_n,Y_1,\ldots,Y_n\}$, is characterized by
\begin{equation}\label{eq:distr}
  \distr_{(\xi,z)} = \ker \left( dz -\frac12 \langle \xi,J d\xi\rangle_{\bR^{2n}} \right).
\end{equation}
For any point $p\in\H^n$,  any vector $v\in \distr_p\subset T_p\bH^n\simeq\bR^{2n+1}$ can be written as $v=(v',v_z)\in \R^{2n}\times \R$. Then, the sub-Riemannian metric is
\begin{equation}
  \langle v,w\rangle_{\bH^n} = \langle v',w'\rangle_{\bR^{2n}}, \qquad \forall v,w\in T_p\H^n, \, p\in\H^n.
\end{equation}
We denote the associated norm by $\|v\|_{\H^n}:=\sqrt{\langle v,v\rangle_{\H^n}}$.

We say that a curve $\gamma:[0,1]\to \H^n$ is \emph{horizontal} if it is absolutely continuous and $\dot\gamma(t)\in\distr_{\gamma(t)}\text{ for a.e.\ } t\in [0,1]$.
The Carnot-Cathéodory distance of $p\in\H^n$ from the origin is then defined as
\begin{equation}
  \delta(p) := \inf\left\{ \int_0^1\|\dot\gamma(t)\|_{\H^n}\,dt \:\bigg|\: \gamma:[0,1]\to\H^n \text{ is horizontal, } \gamma(0)=0\text{, and } \gamma(1)=p \right\}.
\end{equation}
The associated balls are denoted by $B_\varepsilon = \{\delta<\varepsilon\}$.

We now identify a horizontal vector field that plays the role of the ``polar'' vector field $t^{-1}\partial_\varphi$ in Euclidean polar coordinates $(t,\varpi, \varphi)\in \R_+\times \S^{d-2} \times (-\pi/2,\pi/2)\mapsto t(\varpi\cos\varphi,\sin\varphi)\in\R^d$. (See Appendix~\ref{a:euclidean}.) 

\begin{defn}\label{d:polar}
  The \emph{polar vector field} $\ndperp\in \Gamma(\H^1\setminus\cZ)$ is the only unit smooth horizontal vector field orthogonal to $\nd$ and to the generators of rotations around $\cZ$, that is upward pointing on the plane $\{z=0\}$.
\end{defn}

The precise expression in a particular set of coordinates of $\ndperp$ is given in Proposition~\ref{p:ortho}.

\subsection{Polar-like coordinates}
In this section, we introduce suitable polar-like coordinates that will be instrumental in proving our results. 

To present the idea behind these coordinates, let us discuss briefly the case $n=1$. In this setting, for any point $p\in \H^1\setminus\mathcal Z$ there exists a unique length-minimizing geodesic connecting the origin to $p$. The family of all these geodesics depends on two parameters $p_z\in\R$, $ \theta\in\S^1$, corresponding to the curves
\begin{equation}\label{eq:22}
  \gamma_{\theta,p_z}(t)=\left(
    \frac{\cos(tp_z-\theta)-\cos\theta}{p_z},
    \frac{\sin(tp_z-\theta)+\sin\theta}{p_z},
    \frac{tp_z-\sin(tp_z)}{p_z^2}      
\right), \quad t\in (0,2\pi/p_z).
\end{equation}
By definition, it holds that $\dot\gamma_{\theta,p_z}(t)=\nH\delta(\gamma_{\theta,p_z}(t))$ for $t\in (0,2\pi/p_z)$, and one could use \eqref{eq:22} to define polar coordinates $\Psi: \{(t,p_z)\mid 0<tp_z<2\pi\}\times \S^1\to \H^1\setminus\mathcal Z$ by $\Psi(t,\theta,p_z):=\gamma_{\theta,p_z}(t)$. However, this yields a transformation whose jacobian has a complicated form and, moreover, that does not behave well under the anisotropic  dilations $\{\varrho_\lambda\}_{\lambda>0}$. Indeed, although the distance is $1$-homogeneous, geodesics are not invariant under dilations and one can check that $\varrho_\lambda(\Psi(t,\theta,p))=\Psi(\lambda t, \theta,\lambda^{-1}p)$.
For these reasons, we consider a rescaled version of $\Psi$, letting $r=tp_z$. (See Figure~\ref{fig:Phi}.) This yields the map $\Phi:\R_+\times\S^1\times(-2\pi,2\pi)\to \H^1\setminus\mathcal Z$ defined by
\begin{equation}
  \Phi(t,\theta,r):=\left( t\frac { \cos(r-\theta)-\cos\theta}r , t\frac { \sin(r-\theta)+\sin\theta}r , 
    t^2\frac{r-\sin r}{r^2}
  \right).
\end{equation}

\begin{figure}
  \begin{tikzpicture}
  \pgfplotsset{every axis/.append style={
                    axis x line=middle,    
                    axis y line=middle,    
                    axis line style={->}, 
                    xlabel={$|\xi|$},          
                    ylabel={$z$},          
                    xtick distance = {.25},
                    xticklabels = \empty,
                    ytick distance = {0.25},
                    yticklabels = \empty,
            }}
    \begin{axis}[xmin=-.2, xmax=1.2, ymin=-.5, ymax=.5, variable=\x]
            \addplot [domain=-6.28:6.28,samples=60]({2/(\x^2)*(1-cos(\x r))},{(\x-sin(\x r))/(\x^2)});
            \addplot [domain=0:.7,smooth, style={dashed}]({\x},{\x^2});
    \end{axis}
    \draw[style = {->, thick}] (3.7,4.6) -- (3.7-.65,4.6+.15*.65) node[above] {$\Xi$}; 
    \draw[style = {->, thick}] (3.7,4.6) -- (3.7+.15*.65,4.6+.65) node[below right] {$\nH\delta$}; 
    \draw[style={->}] (2.6,3.8) -- node[above left] {$t$} (2.6+.35,3.8+.35) ;
    \draw[style={->}] (5.8,3.8) -- node[above right] {$r$} (5.8-.35,3.8+.35);
    \draw (5,1) node {$\partial B_1$};
\end{tikzpicture}

\caption{Graphical depiction of  $(r,t)$ in $\Phi$ coordinates. Notice that $r$ parametrizes the position on the Carnot-Carath\'eodory unit sphere, up to rotations around the center, while $t$ encodes the distance from the origin. Nevertheless, $\partial_t$ is the dilation vector field, and not $\nH\delta$.}
\label{fig:Phi}
\end{figure}
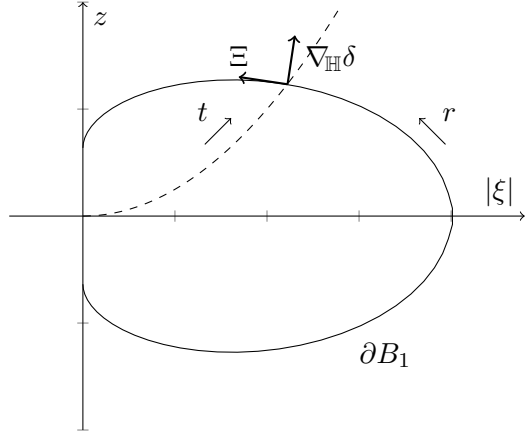

In the general case $n\ge 1$, we denote points on the sphere $\varpi\in\bS^{2n-1}\subset\bR^{2n}$ as $\varpi=(\varpi_1,\ldots,\varpi_n)$, $\varpi_i=(\varpi_i^1,\varpi_i^2)\in \bR^2$, and let $R_r$ be the clockwise rotation of angle $r\in\R$. That is,
\begin{equation}
R_r=\begin{pmatrix}
\cos{r}&-\sin{r}\\
\sin{r}&\cos{r}
\end{pmatrix}.
\end{equation}  
Observe that, in particular, $J=R_{-\pi/2}$.

Then, considering the optimal synthesis for $\H^n$ (see, e.g., \cite{Monroy-Perez1999, Biggs2016}) leads to the following. 

\begin{defn}\label{def:phi-coord}
 Consider $U=\bR_+\times\bS^{2n-1}\times (-2\pi,2\pi)$. Then, the diffeomorphism $\Phi:U\to \H^n\setminus\mathcal Z$ is defined by $\Phi(t,\varpi,r) = (\xi,z),$ where
\begin{equation}
  \xi_i = \frac tr A\,\varpi_i, \quad z = t^2\frac{r-\sin r}{2r^2}, 
  \quad 
  A = -J(\idty-R_r)= \left(
    \begin{array}{cc}
      \sin r & -1+\cos r\\
      1-\cos r & \sin r
    \end{array}
  \right).
\end{equation}
\end{defn}

Henceforth, for any $\varpi\in \bS^{2n-1}$, we represent $T_\varpi \bS^{2n-1} = \varpi^\perp\subset\bR^{2n}$. When clear from the context, we identify vectors $\Theta\in T\bS^{2n-1}$ with $(0,\Theta,0)\in TU$. Thus, $TU\simeq\{(v_t,v_\varpi,v_r)\in\bR^{2n+2}\mid v_\varpi\perp\varpi\}\simeq \bR^{2n+1}$. Moreover, we denote by $d\varpi$ the standard volume measure on $\mathbb{S}^{2n-1}$. Namely, $d\varpi$ stands for  the $(2n-1)$-form 
\begin{equation}\label{eq:Omega}
  \frac{1}{2n}\sum_{k=1}^n \sum_{\ell = 1}^2 (-1)^{\ell-1} \varpi_k^\ell d\varpi_{1}^1\wedge \ldots \wedge \widehat{d\varpi_k^\ell} \wedge\ldots \wedge d\varpi_{n}^2,
\end{equation}
where the hat denotes a missing element.

The next proposition collects some basic facts on $\Phi$, following via direct computations from the explicit optimal synthesis of geodesics in $\H^n$ and \eqref{eq:distr} .

\begin{prop}\label{prop:basic}
  The following hold.
  \begin{enumerate}
    \item[i.] Let $(\xi,z)=\Phi(t,\varpi,r)$. Then, $t = \delta(p)$, i.e., $(\varpi,r)\mapsto \Phi(t,\varpi,r)$ is a parametrization of $\partial B_t\setminus\cZ$. Moreover,
    \begin{equation}\label{eq:psi}
      \frac{|\xi|^2}{z} = \phi(r) := 4\frac{1-\cos r}{r-\sin r}.
    \end{equation}
    \item[ii.] Letting $A'=\partial_rA=R_r$, we have
    \begin{equation}
    D\Phi(t,\varpi,r) =
    \left(
    \begin{array}{ccc}
       \dfrac1r \tilde A\varpi & \dfrac tr \tilde A & -\dfrac t{r^2}\tilde A\varpi+\frac{t}{r}\tilde A'\varpi\\
       t \dfrac{r-\sin r}{r^2} & 0 & -\frac{t^2}{r^2} \frac{r-2 \sin r+r \cos r}{2 r}
    \end{array}
    \right).
  \end{equation}  
    \item[iii.] We have that $\Phi^ *\cL^3 = t^{2n+1}\mu(r)\,dt\,d\varpi\,dr$, where
    \begin{equation}
      \mu(r) = \frac{\left(2-2\cos r-r\sin r\right)(2-2\cos r)^{n-1}}{r^{2n+2}}.
    \end{equation}
  \item[iv.]   We have
  \begin{equation}\label{eq:char_distr}
    \distr_{\Phi(t,\varpi,r)}=\left\{ (v',v_z)\in T_{\Phi(t,\varpi,r)}\bH^n\simeq \bR^{2n}\times\bR \bigg|\: v_z = \frac12\frac tr \left\langle \tilde A\varpi, \tilde Jv' \right\rangle_{\bR^{2n}}  \right\}.
  \end{equation}
  \end{enumerate}
\end{prop}

Henceforth, we let $\{W_3,\ldots,W_{2n}\}\subset \Gamma(\bS^{2n-1})$ be a fixed orthonormal frame for the distribution on $\bS^{2n-1}$ given by $\varpi\mapsto \varpi^\perp\cap (\tilde J\varpi)^\perp\subset T_\varpi\bS^{2n-1}$. 
The following proposition presents the orthonormal frame that is central to all the results of this paper. 

\begin{prop}\label{p:ortho}
%
  An orthonormal frame for $\Phi^*\distr$ is given by $\{V_1,\ldots,V_{2n}\}$, with
  \begin{gather}
    V_1 =\Phi^*(\nH\delta) = \left(1,0,\frac rt \right),\qquad
    V_2 = \Phi^*(\ndperp)=\left(0, \frac rt v(r)\tilde J\varpi,\frac rt w(r)\right),\\
    V_j = \left(0,\frac{|r|}{t\sqrt{2(1-\cos r)}}W_j,0\right), \qquad j=3,\ldots,2n.
  \end{gather}
  Here,
  \begin{equation}
    w(r) = \frac{r}{2-r \cot \left(\frac{r}{2}\right)},
    \qquad 
    v(r)= \frac{r-\sin r}{2-r \sin r-2 \cos r}.
  \end{equation}
\end{prop}

\begin{rem}\label{rmk:rwm}
  All the vector fields above are well defined for $r=0$ and smooth on $U$. In particular,
  \begin{equation}
    \lim_{r\to 0} \frac16\frac{r}{t}w(r) = \lim_{r\to 0} \frac12\frac{r}{t}v(r) = \lim_{r\to 0}\frac{|r|}{t\sqrt{2(1-\cos r)}} =\frac1t, 
  \end{equation}
  Also, we point out that both $v$  and $w$ are odd functions, positive for $r>0$, which explode when $r\to 0$, see Figure \ref{fig:vw}.
\begin{figure}
\includegraphics[scale=.5]{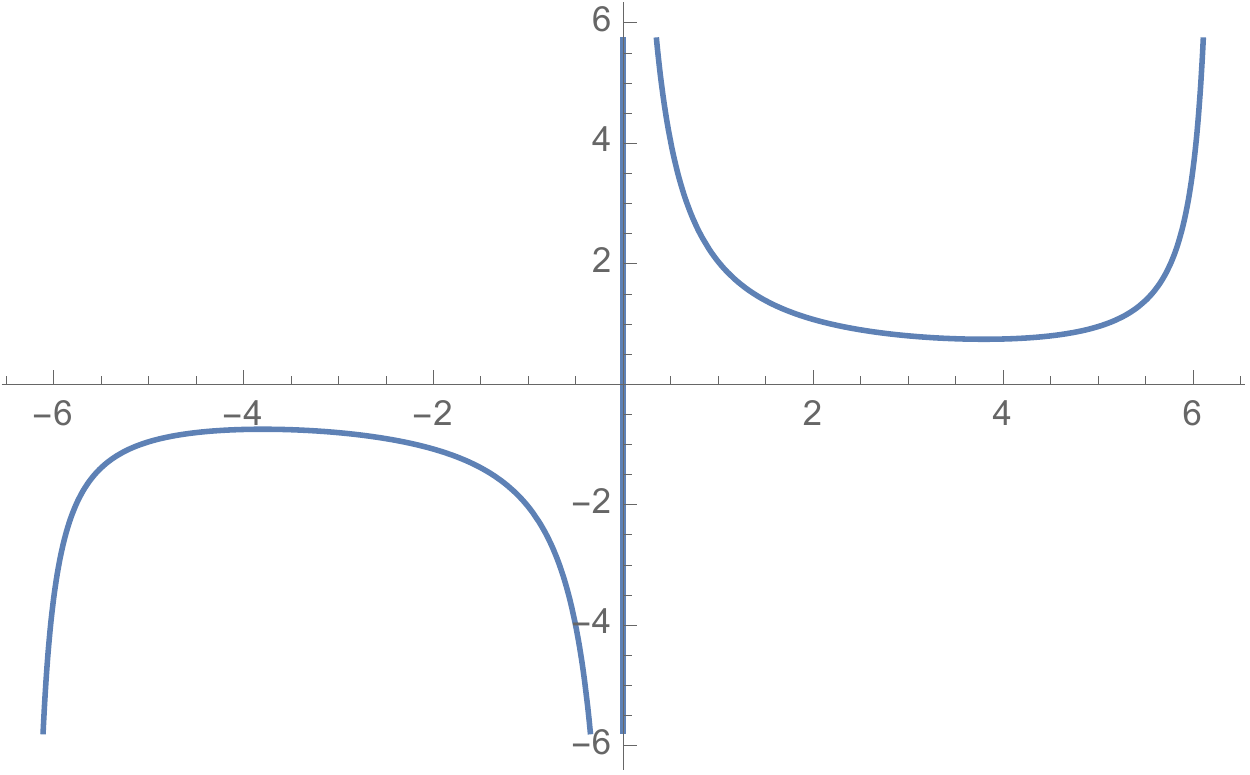} \hspace{1cm} \includegraphics[scale=.5]{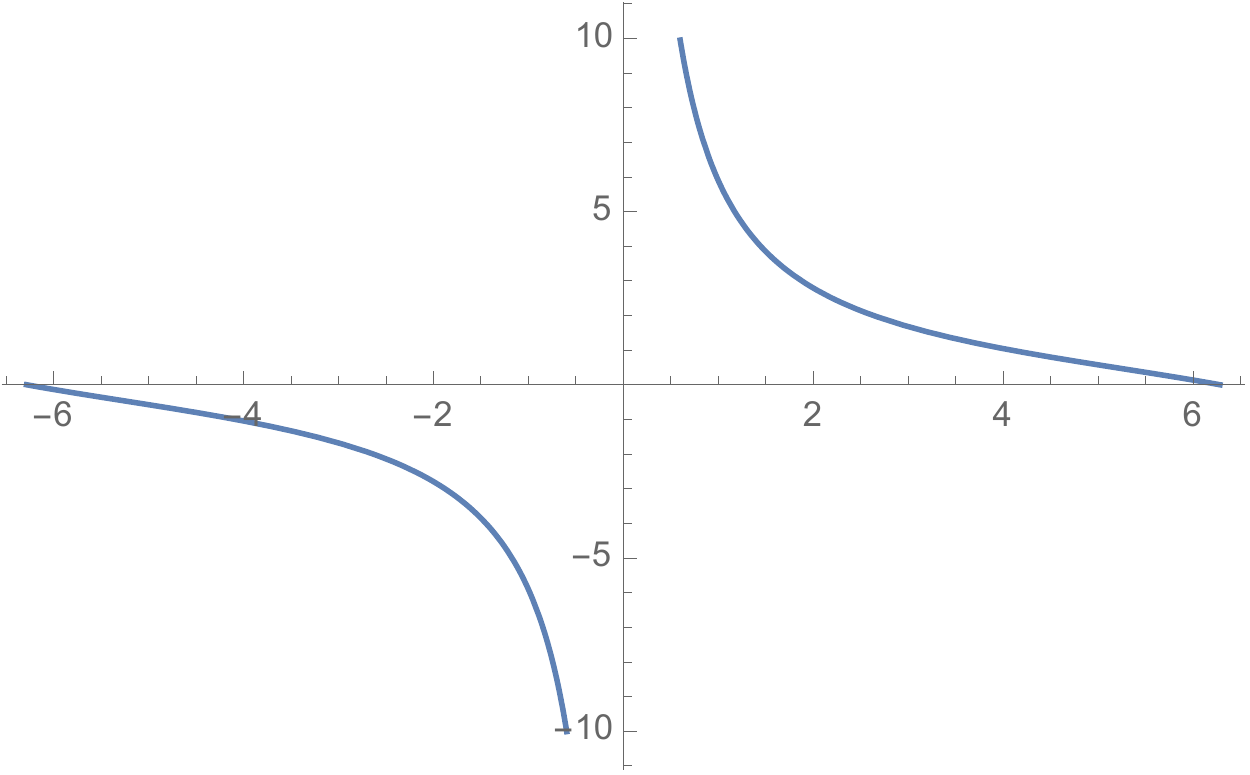}
  \caption{Graphs of the functions $v$ (left) and $w$ (right), defined in Proposition~\ref{p:ortho}.}
  \label{fig:vw}
\end{figure}
\end{rem}

\begin{rem}\label{p:rwm}
  Several important properties of this frame are connected with the following relation, which can be directly checked from the definitions of $w$ and $\mu$:
  \begin{equation}\label{eq:rwm}
    \partial_r(r w \mu) = -nr\mu.
  \end{equation}
\end{rem}

In order to prove Proposition~\ref{p:ortho}, we need the following two preliminary Lemmas.

\begin{lem}\label{lem:pullback-nabla-delta}
  We have
  \begin{equation}
    \phi^*(\nH\delta)|_{\Phi(t,\varpi,r)} = \partial_t + \frac{r}{t}\partial_r.
  \end{equation}
\end{lem}

\begin{proof}
Exploiting the fact that, for fixed $\varpi\in \bS^{2n-1}$ and $p_z\in \bR$, the curves $t\mapsto\Phi(t,\varpi,t p_z)$ are arc-length parametrized geodesics for $0<|t p_z|<2\pi$, one deduces that
  \begin{equation}\label{eq:nabla-delta}
    \nH\delta|_{\Phi(t,\varpi,r)} = \left( \tilde A'\varpi, \frac{t}{r}\frac{1}{2}(1-\cos{r})\right)\in\R^{2n}\times\R.
  \end{equation}

  Observe that
  \begin{gather}
    \left.\Phi_*\partial_t\right|_{\Phi(t,\varpi,r)} = \left(\frac{\tilde A}{r} {\varpi}, t \frac{r-\sin r}{r^2}\right)\in\R^{2n}\times\R,
    \\
    \left.\Phi_*\partial_r\right|_{\Phi(t,\varpi,r)} = \left(\frac{t}r\left({\tilde A'}-\frac{\tilde A}{r}\right) \varpi, -\frac{t^2}{r^2} \frac{r-2 \sin r+r \cos r}{2 r}\right)\in\R^{2n}\times\R.
  \end{gather}
  The statement then follows by direct computations.
\end{proof}

\begin{lem}\label{lem:non-ortho}
  Let $\alpha,\beta:(-2\pi,2\pi)\to\R$ be such that
  \begin{equation}\label{eq:alpha-beta}
    \alpha(r) = \frac{r-\sin r}{r(1-\cos r)}\beta(r), \qquad \forall r\in (-2\pi,2\pi).
  \end{equation}
  Then, letting $\bar V_j=(0,W_j,0)$, $j=3,\ldots, 2n$, a (non-orthonormal) basis for $\Phi^*\distr$ is given by $\{V_1, \bar V_2, \ldots, \bar V_{2n}\}$, with
  \begin{equation}
    V_1 = \Phi^*(\nH\delta) = \left(1,0,\frac rt \right)
    \qquad
    \bar V_2 = \left(0, \frac {\alpha(r)}t J\varpi,\frac {\beta(r)}t\right).
  \end{equation}
\end{lem}

\begin{proof}
By definition, $\Phi_*V_1=\nH\delta$ is a horizontal vector field. Its expression in $\Phi$-coordinates is derived in Lemma~\ref{lem:pullback-nabla-delta}.

Let $V\in\Gamma(\bS^{2n-1})$, i.e., $V\in \bR^{2n}$ such that $V(\varpi)\perp\varpi$. Observe that $A^T J A=2(1-\cos r)J$. Then, by identifying $V \in T\bS^{2n-1}$ and  $(0,V,0)\in TU$ and by \eqref{eq:char_distr}, $\Phi_*V\in\Gamma(\distr)$ if and only if
\begin{equation}
  0=\frac12 \frac tr \left\langle \tilde A \varpi,\frac tr \tilde J\tilde AV(\varpi)\right\rangle_{\bR^{2n}}  \iff V(\varpi)\perp \tilde J\varpi.
\end{equation}
In particular, $\dim(T\bS^{2n-1}\cap \Phi^*\distr)=2n-2$. Thus, the basis $\bar V_3,\ldots, \bar V_{2n}$ for $T\bS^{2n-1}$ yields $2n-2$ additional horizontal directions.

We now turn to show that $\bar V_2$ is linearly independent to the others.
Straightforward computations yield
\begin{equation}
  \Phi_* \bar V_2(t,\varpi,r) = \frac1r \left( \alpha\tilde A \tilde J\varpi + \beta\left(\tilde A'-\frac1r \tilde A\right)\varpi, \beta\frac tr \frac{2\sin r-r\cos r-r}{2r}\right)
\end{equation}
By \eqref{eq:char_distr}, we have that $\Phi_*\bar V_2$ is horizontal if and only if
\begin{equation}\label{eq:perp-hor}
  \beta\frac{(2\sin r-r\cos r-r)}{r} = 
  \left\langle \tilde A\varpi,  \alpha\tilde J\tilde A \tilde J\varpi + \beta\tilde J(\tilde A'-\frac1r \tilde A)\varpi\right\rangle_{\bR^{2n}}.
\end{equation}
Since $\tilde J\varpi\perp\varpi$, we have
\begin{multline}\label{eq:ciao1}
      \left\langle \tilde A\varpi,  \tilde J(\tilde A'-\frac1r \tilde A)\varpi\right\rangle_{\bR^{2n}} 
    =\left\langle \tilde A\varpi,  \tilde J\tilde A'\varpi\right\rangle_{\bR^{2n}} \\
    =\sum_{i=1}^n \left\langle A\varpi_i,  J A'\varpi_i\right\rangle_{\bR^{2}}
    =\sum_{i=1}^n (1-\cos r) |\varpi_i|^2
    = 1-\cos r.
\end{multline}
Here, we used the invariance under rotations in $\bR^2$ of $A$, $J$, and $A'$ and the explicit expression for $(A^TJA')_{11}=1-\cos{r}$. 
Moreover, by the fact that $J^2 = -\idty$ and $(A^TA)_{11}=2(1-\cos r)$, we have
\begin{equation}\label{eq:ciao2}
  \left\langle \tilde A\varpi, \tilde J\tilde A \tilde J\varpi\right\rangle_{\bR^{2n}} = -2(1-\cos r).
\end{equation}
By the assumption on $\alpha$ and $\beta$, the statement follows by \eqref{eq:perp-hor}, \eqref{eq:ciao1}, and \eqref{eq:ciao2}.
\end{proof}

\begin{proof}[Proof of Proposition \ref{p:ortho}]
  Observe that, $\alpha(r) = rv(r)$ and $\beta(r) = rw(r)$ satisfy assumption \eqref{eq:alpha-beta}.
  Firstly, we show that the basis given in Lemma~\ref{lem:non-ortho} can be orthonormalized to  $\{V_1,\ldots, V_{2n}\}$. 
  
  To this purpose, we claim that the push-forward of $\bar V_j$, $j=3,\ldots, 2n$, is orthogonal to vector fields of the form $X=(a,\varpi,br/t)$, with $a,b:(-2\pi,2\pi)\to\R$. Indeed, we have
  \begin{equation}\label{eq:scalar}
    \langle \Phi_* \bar V_j,\Phi_*X \rangle_{\bH^n} = \left\langle \frac tr\tilde A W_j,\frac tr\tilde A \varpi \right\rangle_{\bR^{2n}}+\left\langle\frac{t}{r} \tilde A W_j, \left((a-b)\frac{\tilde A}r +b\tilde A'\right)\varpi\right\rangle_{\bR^{2n}}.
  \end{equation}
  Since $A^TA=2(1-\cos{r})\idty$ and $A^TA'=A$, we have
  \begin{equation}\label{eq:ciao}
    \begin{split}
    \tilde A^T\left((a-b)\frac{\tilde A}r +b\tilde A'\right)\varpi
    &=\left(2(1-\cos{r})\frac{a-b}{r}+b\tilde A\right)\varpi,
    \end{split}
  \end{equation}
Then, by definition of $A$, we get 
\begin{equation}\label{eq:ciao2}
\tilde A\varpi=\left(-\tilde J+\tilde R_r \tilde J\right)\varpi
=(\sin{r})\varpi+(\cos{r}-1)\tilde J\varpi.
\end{equation}
Hence, by \eqref{eq:ciao} and \eqref{eq:ciao2} we get 
\begin{equation}
\tilde A^T\left((a-b)\frac{\tilde A}r +b\tilde A'\right)\varpi\in\operatorname{span}\{\varpi,\tilde J\varpi\}.
\end{equation}
Since $\bar V_j = (0,W_j,0)$ with  $W_j\perp\{\omega,\tilde J\omega\}$,  and $A^T A = 2(1-\cos r)\idty$, this proves the claim.

  As a consequence of the previous claim, we have that $\operatorname{span}\{V_3,\ldots, V_{2n}\}$ is orthogonal to $\operatorname{span}\{V_1,V_2\}$.
  Moreover, for any $i,j=3,\ldots,2n$, we have
  \begin{equation}
    \langle \Phi_* \bar V_i, \Phi_*\bar V_j\rangle_{\H^n} = 2t^2\frac{1-\cos r}{r^2} \langle W_i,W_j \rangle_{\R^{2n}}.
  \end{equation}
  Since $\{W_3,\ldots, W_{2n}\}$ are orthonormal in $\R^{2n}$, this shows that $V_i = \bar V_i/\|\Phi_*\bar V_i\|_\H^n$ and that $\{V_3,\ldots, V_{2n}\}$ is an orthonormal family.
  
  Let us now show that also $\{V_1,V_2\}$ is an orthonormal family. It is clear that, by Lemma~\ref{lem:pullback-nabla-delta}, $\|\Phi_*V_1\|_{H^1} = \|\nd\|_{H^1}\equiv 1$, since $\delta$ satisfies the Eikonal equation. On the other hand, if $\alpha(r)=g(r)\beta(r)$, where $g$ is given by \eqref{eq:alpha-beta}, we have
  \begin{equation}
    |\Phi_*\bar V_2|^2 = \frac{\beta^2}{r^2} \left| \left(g(r)\tilde A\tilde J - \tilde A'-\tilde A/r\right)\varpi \right|_{\bR^{2n}}^2 = \frac{\left(r \cot \left(\frac{r}{2}\right)-2\right)^2}{r^2}, 
  \end{equation}
  thus showing that $\|\Phi_*V_2\|_{\H^n} = w^{-1}\|\Phi_*\bar V_2\|_{\H^n}\equiv1$.
  Finally, we have
  \begin{equation}
    \langle \Phi_*V_1,\Phi_*V_2\rangle_{\bH^n} 
    = \left\langle \tilde A'\varpi, \left(v \tilde A\tilde J+\tilde A'-\frac{\tilde A}r \right)\varpi \right\rangle
    =  \left\langle \tilde A'\varpi, \tilde J\tilde A'\varpi \right\rangle 
    = 0. 
  \end{equation}
  
  To complete the proof, we need to show that $V_2=\Phi^*(\ndperp)$. To this aim, observe that rotations around $\cZ$ are generated by the vector fields $\Phi_*(V_3),\ldots,\Phi_*(V_{2n})$ that are orthogonal to $\Phi_*(V_2)$. Indeed, these are $2n-2$ linearly independent vector fields such that $V(|\xi|^2)\equiv0$, as is evident from the fact that
  \begin{equation}
    \frac{|\xi|^2}2 = \frac{t^2}{r^2}(1-\cos r),\qquad \text{if }(\xi,z)=\Phi(t,\varpi,r).
  \end{equation}
  Finally, a simple computation shows that
  \begin{equation}
    \lim_{r\to 0}dz\left( \Phi_*(V_2) \right)|_{\Phi(t,\varpi,r)} = -\lim_{r\to 0} {t}w(r)\frac{r-2\sin r+r\cos r}{2r^2} = \frac t2>0.
  \end{equation}
  Since $\{z=0\}\setminus\cZ=\Phi(\{r=0\})$ and $\ndperp\perp\Phi^*(V_2)$, this completes the proof.
\end{proof}

\subsection{Horizontal Sobolev spaces}

The Haar measure on $\H^n$ is, up to a constant, the $2n+1$-dimensional Lesbegue measure. This allows to define the space of square integrable functions $L^2(\H^n)$. Moreover, the horizontal gradient associated with the Heisenberg structure is
\begin{equation}
  \nH u = \sum_{i=1}^n \left(  (X_iu)X_i+(Y_iu)Y_i\right), \qquad u\in C^\infty(\H^n).
\end{equation}
Then, $H^1(\H^n)$ is the closure of $C^\infty_c(\H^n)$ w.r.t.\ the horizontal Sobolev norm
\begin{equation}
  \|u\|^2_{H^1(\H^n)} = \int_{\H^n}|u|^2\,dp + \int_{\H^n}|\nH u|^2\,dp.
\end{equation}
The same techniques used in \cite{AFBP} show that $C^\infty_c(\H^n\setminus\{0\})$ is dense in $H^1(\H^n)$. In particular, all the infima appearing in the definitions of the Hardy constants can be calculated on $H^1_c(\H^n)$, i.e., compactly supported functions of $H^1(\H^n)$.

In the following we show how, under suitable technical assumptions, the fact that a certain function belongs to $H^1(\bH^n)$ can be checked in terms of integrals in the $\Phi$ coordinates, regardless of the singularity of the latter. 

\begin{prop}\label{prop:H1}
  Let $u\in L^2(\H^n)$ be such that $u\circ\Phi(t,\varpi,r)=g(t)h(r)$ for some function $g:(0,+\infty)\to \R$ and $h:(-2\pi,2\pi)\to \R$ bounded as $r\to 2\pi$. Then, $u\in H^1(\bH^n)$ if and only if both $\Phi^*u$ and  $\Phi^*|\nH u|$ belong to $L^2(U,t^{2n+1}\mu(r)\,dt\,d\varpi\,dr)$.
\end{prop}

\begin{proof}
  The necessary part of the proof is immediate. We thus focus on the other implication. To this purpose, observe that $\Phi^*u\in L^2(U,t^{2n+1}\mu(r)\,dt\,d\varpi\,dr)$ is trivially equivalent to $u\in L^2(\H^n)$. Fix then $u\in L^2(\H^n)$ such that $\Phi^*|\nH u|\in L^2(U,t^{2n+1}\mu(r)\,dt\,d\varpi\,dr)$.

  By assumption, there exists $\varphi:(0,+\infty)\to \R$ such that $u(\xi,\alpha|\xi|^2)=h(\phi^{-1}(\alpha))\varphi(|\xi|)$ for $\alpha>0$ and where $\phi$ is defined in \eqref{eq:phi}. Fix a sequence of positive numbers $\alpha_k\rightarrow0$ and let $\Omega_k =\{ |\xi|^2\le \alpha_k |z| \}$. Define
  \begin{equation}
    v_k(\xi,z) = h\circ\phi^{-1}(\alpha_k)\varphi(\sqrt{\alpha_k |z|}), \qquad (\xi,z)\in\Omega_k.
  \end{equation}
  The Euclidean gradient of $v_k$ can be directly computed as
  \begin{equation}
    |\nabla v_k| = |\partial_z v_k| = h\circ\phi^{-1}(\alpha_k)\sqrt{\frac{\alpha_k}{4|z|}}\,\big|\varphi'(\sqrt{\alpha_k|z|})\big|.
  \end{equation}
  Thus we get
  \begin{equation}
    \int_{\Omega_k} |\nH v_k|^2\,dp \le \left(h\circ\phi^{-1}(\alpha_k)\right)^2\,\omega_{2n} \alpha_k\int_0^{+\infty} |\varphi'(\eta)|^2\eta^{2n-1}\,d\eta.
  \end{equation}
  Since $\Phi^*|\nH u|\in L^2(U,t^{2n+1}\mu(r)\,dt\,d\varpi\,dr)$, the integral on the r.h.s.\ is bounded, as is the quantity $h\circ\phi^{-1}(\alpha_k)$ as $k\to +\infty$. Therefore, 
  \begin{equation}
    \lim_{k\to +\infty}\int_{\Omega_k} |\nH v_k|^2\,dp  = 0.
  \end{equation}
  
  Define $u_k = u|_{\Omega_k^c} + v_k|_{\Omega_k}$. Thanks to a triangle inequality argument, the above result and the fact that $\Phi^*|\nH u|$ belongs to $L^2(U,t^{2n+1}\mu(r)\,dt\,d\varpi\,dr)$ imply that $(u_k)_k$ is a Cauchy sequence in $H^1(\H^1)$. Since, by construction, $u_k\rightarrow u$ pointwise, this completes the proof of the statement.
\end{proof}

\section{Upper bounds of Hardy constants on $\bH^n$}\label{sec:yang}

In this Section, we prove Theorem~\ref{thm:counter-yang}. 
We start by considering the radial Hardy constant, in the following. 

\begin{prop}\label{p:radial-const}
  It holds $c_n^\text{rad}=0.$
\end{prop}

\begin{proof}
Let $h\in C^\infty_c((0,+\infty))$ and define 
\begin{equation}
  u\circ\Phi(t,\varpi,r) = \left(\frac{r}{t}\right)^nh(t).
\end{equation}
By continuity, the above defines a continuous function $u:\H^1\to \R$.
Then, direct computations yield
\begin{equation}
  \langle\nH u,\nH\delta\rangle|_{\Phi(t,\varpi,r)} = \left( \frac rt\right)^n h'(t),\quad
  \langle\nH u,\ndperp\rangle|_{\Phi(t,\varpi,r)} = n\left( \frac rt\right)^n w(r)\frac{h(t)}t.
\end{equation}
In particular, this implies that $u\in H^1_0(\H^1)$ by Proposition~\ref{prop:H1}. Finally, direct computations yield
\begin{equation}\label{eq:boh}
  c^{\text{rad}}_n\le \frac{\int_{\H^n}|\langle\nH u,\nH\delta\rangle|^2\,dp}{\int_{\H^n}\frac{u^2}{\delta^2}\,dp}=\frac{\int_0^{+\infty} |h'|^2t\,dt}{\int_0^{+\infty} \frac{h^2}{t^2}t\,dt}.
\end{equation}
Observe that, letting $v:\R^2\to \R$ be the radially symmetric function defined by $v(p)=h(|p|)$, where $|p|$ is the Euclidean norm of $p$, we have
\begin{equation}
  \frac{\int_0^{+\infty} |h'|^2t\,dt}{\int_0^{+\infty} \frac{h^2}{t^2}t\,dt} = \frac{\int_{\R^2} \left|\nabla v\right|\,dp }{\int_{\R^2} \frac{|v|^2}{|p|^2}\,dp}.
\end{equation}
Since the Euclidean Hardy constant in $\R^2$ is obtained via radially symmetric functions, taking the infimum w.r.t.\ $h\in C^\infty_c((0,+\infty))$ in \eqref{eq:boh} yields the statement by \eqref{eq:hardy-eucl}.
\end{proof}

\begin{rem}
  The proof is based on the fact that all functions of the form $f\circ\Phi(t,\varpi,r)=\varphi(r/t)$ satisfy $\langle\nH f,\nH\delta\rangle\equiv0$. 
\end{rem}

We now turn our attention to the full Hardy constant.

\begin{prop}\label{p:koranyi}
  For any $n\ge 1$ we have that $c_n<n^2$
\end{prop}

\begin{proof}
  Recall that the Koranyi norm associated with $\H^n$ is $N=(|\xi|^4+16z^2)^{1/4}$. By Proposition~\ref{prop:basic}(i), we then have
  \begin{equation}
    N\circ\Phi(t,\varpi,r) = \frac{\sqrt{2}t}{|r|}\sqrt[4]{r^2-2r\sin r-2\cos r+2}.
  \end{equation}
With a little abuse of notation we still denote by $N$ the Korany norm in the coordinates $\Phi$.
Since $\delta(\Phi(t,\cdot,\cdot))=t$, $t>0$, for any $\alpha\in\bR$ we have 
\begin{equation}\label{eq:k-pot}
	\frac{|N^{\alpha/2}|^2}{\delta^2} =
2^{\alpha /2}t^{\alpha-2}\left(\frac{\sqrt[4]{r^2-2 r \sin
   (r)-2 \cos (r)+2}}{|r|}\right)^{\alpha } = t^{\alpha-2}\gamma(r)^\alpha.
\end{equation}
Here, $\gamma:[-2\pi,2\pi]\to\R$ is defined by the last equality. Observe that the above is independent of $n$ and $\varpi$. Then, using the orthonormal basis of Proposition~\ref{p:ortho}, we obtain
\begin{equation}
  |\nH N(t,\varpi,r)|^2=\frac{1-\cos r}{\sqrt{2+r^2-2\cos r-2r\sin r}}.
\end{equation}
In particular, for any $\alpha\in\bR$ we have
\begin{equation}\label{eq:k-grad}
	\begin{split}
	|\nabla_\H(N^{\alpha/2})|^2 
&=\frac{\alpha^2}{4}\frac{r^2 (1-\cos r)}{2 \left(r^2-2 r \sin (r)-2 \cos
   (r)+2\right)} \frac{|N^{\alpha/2}|^2}{\delta^2}\\
&=\frac{\alpha^2}{4}t^{\alpha-2}\gamma(r)^\alpha\eta(r).
	\end{split}
\end{equation}
Here, $\eta$ is defined by the last equality, and is independent of $\alpha$.

Observe that both $\gamma$ and $\eta$ are non-negative continuous function. Since one can check that $\gamma\ge1/\sqrt{\pi}$, $\gamma^\alpha\eta$ is integrable w.r.t.\ $\mu(r)\,dr$ for any $\alpha\in\R$.
In particular, this implies that $|N^{\alpha/2}|^2\delta^{-2}$ and $|\nH(N^{\alpha/2})|^2$ are integrable on $[1,+\infty)\times(-2\pi,2\pi)$ w.r.t.\ $t^{2n+1}\mu(r)\,dt\,dr$ if and only if $\alpha<-2n$. 

Now, let us fix a smooth function $\chi:\mathbb R_+\to [0,1]$ such that $\chi|_{[0,1/2]}\equiv0$ and $\chi|_{[1,+\infty]}\equiv 1$. Then, for $\alpha<-2n$, we let
\begin{equation}
  u_\alpha\circ\Phi(t,\theta,r) = 
  \begin{cases}
    \chi(t) N^{\alpha/2}(1,\theta,r), &\qquad \text{ if } t\le 1,\\
    N^{\alpha/2}(t,\theta,r), &\qquad \text{ otherwise}.\\
  \end{cases}
\end{equation}
Then $u_\alpha$ can be extended by continuity to the whole $\H^n$.
By definition of $\chi$, \eqref{eq:k-pot}, and \eqref{eq:k-grad}, for any $\alpha<-2n$ there exists $(v_k)_k\subset C^\infty_c(\H^1)$ such that
\begin{equation}
  \lim_{k\to +\infty}\int_{\H^n} \frac{|v_k|^2}{\delta^2}\,dp = \int_{\H^n} \frac{|u_\alpha|^2}{\delta^2}\,dp ,
  \qquad
  \lim_{k\to +\infty}\int_{\H^n} |\nH v_k|^2\,dp = \int_{\H^n} |\nH u_\alpha|^2\,dp.
\end{equation}
In particular, by Proposition~\ref{prop:H1}, we have
\begin{equation}
  c_n \le \inf \left\{ \frac{\int_{\H^1} |\nH u_\alpha|^2\,dp}{\int_{\bH^1} \frac{u_\alpha^2}{\delta^2}} :\: \alpha\in[-2n+1,-2n)\right\}.
\end{equation}
Let us estimate the quotient above.
By \eqref{eq:k-pot}, we have
\begin{equation}\label{eq:u-den}
  \int_{\H^n}\frac{|u_\alpha|^2}{\delta^2}\,dp \ge \int_{\delta\ge 1}\frac{|u_\alpha|^2}{\delta^2}\,dp = |\S^{2n-1}|\int_{1}^{+\infty} t^{\alpha+2n-1}\,dt\int_{-2\pi}^{2\pi}\gamma^\alpha\mu\,dr.
\end{equation}
Observe that the integral in $t$ on the r.h.s.\ goes to $+\infty$ as $\alpha\to (-2n)^-$.
Moreover, $N^{\alpha/2}|_{t=1}$ and $\partial_r(N^{\alpha/2})|_{t=1}$ are uniformly bounded from above for $\alpha\in[-2n+1,-2n)$. As a consequence, there exists a constant $C>0$ such that $|\nabla_\bH u_\alpha|^2\le C$ on $\{\delta\le 1\}$. In particular, by \eqref{eq:k-grad}, we obtain
\begin{equation}\label{eq:u-num}
  \int_{\H^1}{|\nH u_\alpha|^2}\,dp \le C \mathcal L^3(\{0\le \delta\le 1\}) + \frac{\alpha^2}4|\S^{2n-1}| \int_{1}^{+\infty} t^{\alpha+2n-1}\,dt \int_{-2\pi}^{2\pi}\gamma^\alpha\eta\mu\,dr.
\end{equation}
Taking the quotient of \eqref{eq:u-num} and \eqref{eq:u-den}, and passing to the limit as $\alpha\to -2n$, yields
\begin{equation}\label{eq:ch}\begin{split}
  c_n\le n^2\frac{ \int_{-2\pi}^{2\pi}\gamma^{-2n}\eta\mu\,dr}{\int_{-2\pi}^{2\pi}\gamma^{-2n}\mu\,dr}.
  \end{split}
\end{equation}
Here, we passed to the limit under the integral signs by dominated convergence.
Simple computations show that $\eta(0)=1$, $\eta(\pm2\pi)=0$, and that $\eta$ is monotone decreasing in $|r|$. Hence, for any $a>0$, it holds
\begin{equation}
\int_{|r|>a}\gamma^{-2n}\eta\mu\,dr < \eta(a)\int_{|r|>a}\gamma^{-2n}\mu\,dr,
\quad\text{and}\quad
\int_{|r|\le a}\gamma^{-2n}\eta\mu\,dr \le \int_{|r|\le a}\gamma^{-2n}\mu\,dr.
\end{equation}
Since $\eta(a)<1$ and $\int_{|r|>a}\gamma^{-2n}\mu\,dr>0$, together with \eqref{eq:ch}, the above yields the statement.
\end{proof}
\begin{rem}
  The proofs of Propositions~\ref{p:radial-const} and \ref{p:koranyi} are obtained by considering two different sequences of functions. It is interesting to note that it does not seem possible to build a single sequence yielding both bounds at the same time.
\end{rem}

\section{Non-radial Hardy inequalities on homogeneous cones}\label{sec:hardy-vert}

In this section we prove Theorems~\ref{thm:hardy-intro1} and \ref{thm:hardy-intro2}. To this aim we need the following.

\begin{lem}\label{prop:null_div}
Let $V\in\Gamma(\H^ n\setminus\cZ)$ be given by $\Phi^ * V(r,\varpi,t)=\varphi(r,t) \tilde J\varpi$ where $\varphi:\R_+\times (-2\pi,2\pi)\to \R$ is integrable w.r.t.\ $t^ {2n+1}\mu\,dtdr$. Then, for any $f\in C^\infty_c(\H^ n)$ it holds
\begin{equation}
\int_{\H^n} Vf\,dp = 0.
\end{equation}
\end{lem}
\begin{proof}
Let $\mathcal V\in\Gamma(U)$ be defined as $\mathcal V(t,\varpi,r)=(0,\tilde J\varpi,0)$. Then,
  \begin{equation}\label{eq:vale}
  \int_{\R^{2n+1}} Vf\;dp=  \int_0^\infty\varphi(r,t) t^{2n+1}\int_{-2\pi}^{2\pi}\mu(r)\int_{\mathbb S^{2n-1}} \mathcal V (\Phi^*f)\;d\varpi dt dr.
  \end{equation}
By the divergence theorem we get 
\begin{equation}
\int_{\mathbb S^{2n-1}} \mathcal V(\Phi^*f)\;d\varpi=-\int_{\mathbb S^{2n-1}}(\Phi^*f)\operatorname{div}_{\bS^ {2n-1}} \mathcal V\;d\varpi.
\end{equation}
Thus, in order to prove the statement it suffices to show that $\operatorname{div}_{\bS^ {2n-1}} \mathcal V\equiv 0$.

Henceforth, with abuse of notation, we denote the volume form on $\mathbb S^{2n-1}$, as defined in \eqref{eq:Omega}, by $\Omega$. 
The restriction of the vector field $\mathcal V$ to $\Gamma(\mathbb S^{2n-1})$, still denoted by $\mathcal V$, reads
\begin{equation}
\mathcal   V(\varpi) = (\varpi_1^2, -\varpi_1^1, \ldots, \varpi_n^2,-\varpi_n^1)=\sum_{k=1}^{n}\sum_{\ell=1}^2(-1)^{\ell-1}\varpi_k^{\ell+1}\partial_{\varpi_k^\ell},
\end{equation}
where by convention we let $\ell+1 = \ell+1\mod 2$. 
Recall that, by definition, $(\operatorname{div}_{\bS^{2n-1}}V)\Omega = d(\iota_{\mathcal V}\Omega)$.
Since $d \varpi_{k}^\ell(\mathcal V) = (-1)^{\ell-1} \varpi_{k}^{\ell+1}$, we have
\begin{equation}
  \alpha_k:=\iota_{\mathcal V} (d\varpi_k^1\wedge d\varpi_k^2) = \varpi_k^1 d\varpi_k^1 + \varpi_k^2 d\varpi_k^2.
\end{equation}
Henceforth we let, for simplicity, $d\varpi_k := d\varpi_k^1\wedge d\varpi_k^2$.
By the properties of the contraction operator we can then compute:
\begin{equation}
  \begin{split}
      \iota_{\mathcal V}\big( d\varpi_{1}^1\wedge & \ldots \wedge \widehat{d\varpi_k^\ell} \wedge\ldots \wedge d\varpi_{n}^2 \big) \\
      &= \iota_{\mathcal V}\left( d\varpi_{k}^{\ell+1} \wedge d\varpi_{1}\wedge\ldots \wedge \widehat{d\varpi_k} \wedge\ldots \wedge d\varpi_{n}  \right)\\
      &\qquad- d\varpi_k^{\ell+1}\wedge \iota_{\mathcal V}( d\varpi_{1}\wedge\ldots \wedge \widehat{d\varpi_k} \wedge\ldots \wedge d\varpi_{n})\\
      &= (-1)^{\ell} \varpi_k^\ell d\varpi_{1}\wedge\ldots \wedge \widehat{d\varpi_k} \wedge\ldots \wedge d\varpi_{n} \\
      &\qquad- \sum_{i\neq k}d\varpi_k^{\ell+1} \wedge d\varpi_1\wedge\ldots\wedge \alpha_i \wedge\ldots \wedge\widehat{d\varpi_k}\wedge\ldots\wedge d\varpi_n
  \end{split}
\end{equation}
Thus, we obtain $\iota_{\mathcal V}\Omega=-(A+B)/2n$, where
\begin{gather}
  A = \sum_{k=1}^n\sum_{\ell = 1}^2 (\varpi_k^\ell)^2 d\varpi_{1}\wedge\ldots \wedge \widehat{d\varpi_k} \wedge\ldots \wedge d\varpi_{n},\\
  B= \sum_{k=1}^n\sum_{\ell = 1}^2 \sum_{i\neq k} (-1)^{\ell-1}  \varpi_k^{\ell}d\varpi_k^{\ell+1} \wedge d\varpi_1\wedge\ldots\wedge \alpha_i \wedge\ldots \wedge\widehat{d\varpi_k}\wedge\ldots\wedge d\varpi_n.
\end{gather}

Let us now compute $d(\iota_{\mathcal V}\Omega)$. By the properties of $d$ and since $d\alpha_i = 0$, we have
\begin{gather*}
  dA = \sum_{k=1}^n\sum_{\ell = 1}^2 2\varpi_k^\ell d\varpi_k^\ell\wedge d\varpi_{1}\wedge\ldots \wedge \widehat{d\varpi_k} \wedge\ldots \wedge d\varpi_{n} = 2\sum_{k=1}^n d \varpi_1\wedge \ldots\wedge \alpha_k \wedge\ldots\wedge d\varpi_n,\\
  dB
  = 2(n-1)\sum_{i=1}^n d \varpi_1\wedge \ldots\wedge \alpha_i \wedge\ldots\wedge d\varpi_n
\end{gather*}
Thus,
\begin{equation}
  \begin{split}
  d\left(\iota_{\mathcal V}\Omega\right) 
  &=-{2}\sum_{i=1}^n d \varpi_1\wedge \ldots\wedge \alpha_i \wedge\ldots\wedge d\varpi_n\\
  \end{split}
\end{equation}

Observe that we need to compute the above $2n-1$ form on vectors tangent to the sphere. For any $v\in T_\varpi\bS^{2n-1}$, we have $\sum_{i=1}^n\alpha_i(v)= \langle\varpi,v\rangle_{\R^{2n}}=0$, which yields
\begin{equation}
  \alpha_i|_{T\bS^{2n-1}} = -\sum_{j\neq i} \alpha_j|_{T\bS^{2n-1}}.
\end{equation}
Together with the fact that $\alpha_j\wedge d\varpi_j=0$, this implies that
\begin{equation}
  d\left(\iota_{\mathcal V}\Omega\right)|_{T\bS^{2n-1}} = {2}\sum_{i=1}^n\sum_{j\neq i} d \varpi_1\wedge \ldots\wedge \underbrace{\alpha_j}_{\text{$i$-th position}}  \wedge\ldots\wedge d\varpi_n = 0,
\end{equation}
completing the proof of the statement.
\end{proof}

\begin{proof}[Proof of Theorem~\ref{thm:hardy-intro2}]
  Let $f\in C^\infty_c(\H^n)$. By Lemma~\ref{prop:null_div} and the expression of $\Xi$ in $\Phi$-coordinates given in Proposition~\ref{p:ortho}, we get
  \begin{equation}
    \int_{\H^n} \Xi f\,dp 
          = \int_{0}^{+\infty} \int_{\bS^{2n-1}}\left(\int_{-2\pi}^{2\pi} \frac{rw(r)}{t}(\partial_r f\circ\Phi)|_{(t,\varpi,r)}\mu(r)\,dr\right)\,d\varpi\,t^{2n+1}\,dt\\
  \end{equation}
  Then, by Remarks~\ref{rmk:rwm} and \ref{p:rwm}, an integration by parts yields
  \begin{equation}
          \int_{\H^n} \Xi f\,dp 
          = n\int_{0}^{+\infty} \int_{\bS^{2n-1}}\left(\int_{-2\pi}^{2\pi} \frac rt f\circ\Phi(t,\varpi,r)\mu(r)\,dr\right)\,d\varpi\,t^{2n+1}\,dt
          = n\int_{\H^n} \frac{f}{\delta}\psi\,dp.
  \end{equation}
  Here, no boundary terms appear since $\lim_{r\to\pm2\pi}w(r)=0$.
  By density the above holds for any Lipschitz function compactly supported outside the origin. In particular, letting $f = u^2/\delta$ where $u\in C^ \infty_c(C_\Sigma)$, the Cauchy-Schwarz inequality implies
   \begin{equation}
    \frac n2\int_{C_\Sigma} \frac{u^2}{\delta^2}\psi\,dp \le \int_{C_\Sigma} \frac{|u|}{\delta}|\Xi u|\,dp \le \left(\int_{C_\Sigma} \frac{u^2}{\delta^2}\psi\,dp\right)^{1/2}\left(\int_{C_\Sigma} \frac{|\Xi u|^2}{\psi}\,dp\right)^{1/2}.
  \end{equation}
  Here we used that $C_\Sigma \subset\{\psi>0\}$.
  By construction, we have that $|\Xi u|^2\le |{\nH}^\perp u|^2$, and thus the above yields \eqref{eq:hardy2}.
  
  We now turn to the proof of the sharpness. Recall that, in $\Phi$-coordinates, there exists $\rho_\Sigma\in(0,2\pi)$ such that $C_\Sigma = \{\Phi(t,\varpi,r)\mid r>\rho_\Sigma\}$. Let $\rho>\rho_\Sigma$ and $\eta>0$ be sufficiently small, and consider a cut-off function $\chi:(\rho_\Sigma,2\pi)\to [0,1]$ such that
  \begin{equation}
    \chi|_{(\rho_\Sigma,\rho)}\equiv 0, \qquad \chi|_{(\rho+\eta,2\pi)}\equiv 1.
  \end{equation}
  Moreover, consider $0<t_1<t_2<+\infty$ and fix a function $\varphi\in C_c^\infty((0,+\infty))$ with $\supp\varphi\subset[t_1,t_2]$. Let $\gamma>-1/2$ and define
  \begin{equation}
    v(r):= \chi(r) \left( rw(r)\mu(r) \right)^\gamma, 
    \qquad
    u\circ\Phi(t,\varpi,r) = \varphi(t) v(r).
  \end{equation}
  The above definition for $u$ can be extended by continuity to the whole $\H^n$, since $\lim_{r\to2\pi}v(r)=0$. Observe that $v \in L^2([\rho_\Sigma,2\pi],\mu\,dr)$ for any $\gamma>-1/2$. Indeed,
  \begin{equation}\label{eq:taylor}
  	v^2\mu \sim  (2\pi)^{-(2n+1)(2\gamma+1)}\pi^{2\gamma} (2\pi-r)^{2n(2\gamma+1)-1}
	\quad\text{ as }r\to(2\pi)^-.
  \end{equation}
 Moreover, we have $\supp u\subset B_{t_2}\setminus B_{t_1}$ and $|\nH^\perp u|^2 = |\Xi u|^2$, so that, by Proposition~\ref{prop:H1}, 
  \begin{equation}\label{eq:sharp-proof}
    c_n^\perp(\Sigma,\psi)\le R_u:= \frac{\int_{C_\Sigma} \psi^{-1}|\Xi u|^2\,dp}{\int_{C_\Sigma} \psi\frac{u^2}{\delta^2}\,dp}.
  \end{equation}

  By Remark~\ref{rmk:rwm}, it holds that $v'= (rw\mu)^\gamma(\chi'-n\gamma\chi w^ {-1})$. Thus, by the $\Phi$-coordinate expression of $\Xi$ given in Proposition~\ref{p:ortho}, we have
  \begin{equation}
    |\Xi u|\circ\Phi = 
    \begin{cases}
    	0 & \text{ for }  r\in [\rho_\Sigma,\rho),\\
	{\varphi}{\delta^{-1}}r(rw\mu)^{\gamma}(\chi'w-n\gamma\chi)& \text{ for }  r\in [\rho,\rho+\eta),\\
	n\gamma r{|u|}{\delta^{-1}}& \text{otherwise}.
    \end{cases}
  \end{equation}
Recalling that $\psi\circ\Phi=r$ and that $dp = t^{2n-1}\mu \,d\varpi\,dt\,dr$, we get
  \begin{equation}\label{eq:ru-sharp}
      R_u \le \frac{\int_{\rho}^{\rho+\eta}r(rw\mu)^{2\gamma}(\chi'w-n\gamma\chi)^2 \mu\,dr}{\int_{\rho_\Sigma}^{2\pi}  v^2\psi \,\mu\,dr } + n^2\gamma^2.
  \end{equation}
  Observe that there exists $C>0$ such that, for $\gamma\ge -1/2$, we have
  \begin{equation}
  	\int_{\rho}^{\rho+\eta}(rw\mu)^{2\gamma}(\chi'w-n\gamma\chi)^2 \,\psi^{-1}\mu\,dr \le Ce^\gamma\eta,.
  \end{equation}
  The statement then follows by \eqref{eq:sharp-proof} and \eqref{eq:ru-sharp}. Indeed, thanks to \eqref{eq:taylor}, it holds
  \begin{equation}
    \lim_{\gamma\to (-1/2)^+}\int_{\rho_\Sigma}^{2\pi}  v^2\psi \,\mu\,dr = +\infty.\qedhere
  \end{equation}
\end{proof}

\begin{proof}[Proof of Theorem~\ref{thm:hardy-intro1}]
  Observe that $C_\Sigma = \{\rho_\Sigma< \psi\le 2\pi\}$. Thus, for any $u\in C^\infty_c(C_\Sigma)$, it holds
  \begin{equation}
    \rho_\Sigma^2\frac{\int_{C_\Sigma} \psi^{-1}|\nH^\perp u|^2\,dp}{\int_{C_\Sigma} \psi\frac{u^2}{\delta^2}\,dp} 
    < \frac{\int_{C_\Sigma} |\nH^\perp u|^2\,dp}{\int_{C_\Sigma} \frac{u^2}{\delta^2}\,dp}
    \le 4\pi^2 \frac{\int_{C_\Sigma} \psi^{-1}|\nH^\perp u|^2\,dp}{\int_{C_\Sigma} \psi \frac{u^2}{\delta^2}\,dp}
  \end{equation}
  By definition of $c_n^\perp(\Sigma)$ and $c_n^\perp(\Sigma,\psi)$, thanks to Theorem~\ref{thm:hardy-intro2} taking the infimum for $u\in C^\infty_c(C_\Sigma)$ in the above yields the statement.
\end{proof}

\section{Full Hardy constant on homogeneous cones}\label{sec:santalo}

Let $\Sigma\subset \S_{+}^{2n}$ be a spherical cap. Recall that the associated homogeneous cone $C_\Sigma$ is uniquely identified by a parameter $\alpha_\Sigma>0$ such that $C_\Sigma=\{(x,y,z)\in\R^{2n+1}\mid |x|^2+|y|^2<\alpha_\Sigma z\}$. This section is devoted to the proof of the following result on the full Hardy constant $c_n(\Sigma)$, defined in \eqref{eq:full-hardy}. Its proof is based on a consequence of the sub-Riemannian Santal\'o formula, presented in \cite{prandi2019}.

\begin{thm}\label{thm:santalo}
  Let $\Sigma\subset \S_{+}^{2n}$ be a spherical cap. Then,
  \begin{equation}
      c_n(\Sigma) \ge \frac{n}{\alpha_\Sigma}\frac{16\pi^3}{16+\alpha_\Sigma^2}.
  \end{equation}
  In particular, $c_n(\Sigma)\rightarrow+\infty$ as $\Sigma$ degenerates to a point.
\end{thm}

\begin{proof}
For any $p\in C_\Sigma$, we identify $\distr_p\subset T_p\bH^n$ with the corresponding plane through $p$ in $\H^n$. Recall that any straight line $\gamma:\R\to\H^n$ through $p$ and contained in $\distr_p$ is a sub-Riemannian length-minimizer from $p$ to any $\gamma(t)$, $t\in \R$. We henceforth denote by $m(p)$ the maximal length of any such geodesic intersected with $C_\Sigma$.

  Let $u\in C^\infty_c(C_\Sigma)$. We apply \cite[Proposition~2, Eq.~(5)]{prandi2019} and the natural reduction procedure given by \cite[Example~2]{prandi2019}, to get
  \begin{equation}\label{eq:santalo}
    \int_{C_\Sigma} |\nH u|^2\,dp
    \ge {2n\pi^2} \int_{C_\Sigma}   \frac{|u|^2}{m^2}\,dp.
  \end{equation}
  Indeed, $m(p)$ is the maximum of the quantity $L$ appearing in \cite[Proposition~2]{prandi2019}. Observe that \cite[Proposition~2]{prandi2019} is valid under a compactness assumption on the ambient space. This can be settled by applying it to any compact set $\Omega\subset C_\Sigma$ with $\supp u \subset \Omega$, since the quantity $L$ computed w.r.t.\ $\Omega$ is bounded from above by the same quantity computed on $C_\Sigma$.

 Let $\tau_{p^{-1}}(q) =  p^{-1}\star q$, $q\in\H^n$, be the left translation w.r.t.\ the Heisenberg group law \eqref{eq:group}.
 By left-invariance of the sub-Riemannian structure on $\H^n$, it holds that for any straight line $\gamma:\R\to\H^n$ contained in $\distr_p$ and such that $\gamma(0)=p$ the curve $\eta=\tau_{p^{-1}}\circ\gamma$ is such that  $\eta(0)=0$ and $\eta(t)\in \distr_0= \{z=0\}$ for any $t\in\R$. In particular,  $m(p)$ coincides with the maximal length of such $\eta$ intersected with $\tau_{p^{-1}}(C_\Sigma)$. Since the sub-Riemannian length of such $\eta$ coincides with the Euclidean one (this can be checked, e.g., via the explicit expression of the geodesics given in Definition~\ref{def:phi-coord}), we have
 \begin{equation}\label{eq:mp}
 	m(p)\le \diam_{\text{eucl}}(\tau_{p^{-1}}(C_\Sigma)\cap \{z=0\}).
 \end{equation}

Letting $p=(x_0,y_0,z_0)$, straightforward computations yield
 \begin{multline}
 \tau_{p^{-1}}(C_\Sigma) = \bigg\{ (x,y,z)\in \R^{2n+1}\mid \left| x+x_0+\frac{\alpha_\Sigma}{4}y_0 \right|^2 +\left| y+y_0-\frac{\alpha_\Sigma}{4}x_0 \right|^2 \\
 < \alpha_\Sigma |z+z_0|^2 +\frac{\alpha_\Sigma^2}{16} (|x_0|^2+|y_0|^2)\bigg\}.
 \end{multline}
 Using this in \eqref{eq:mp}, since $|x_0|^2+|y_0|^2<\alpha_\Sigma z_0$, we obtain
 \begin{equation}\label{eq:mp2}
 	m(p)\le 2\sqrt{\alpha_\Sigma z_0^2 +\frac{\alpha_\Sigma^2}{16} (|x_0|^2+|y_0|^2)}\le \frac{\sqrt{z_0}}2\sqrt{\alpha_\Sigma (16+\alpha_\Sigma^2)}.
 \end{equation}
 The statement then follows by \eqref{eq:santalo}, \eqref{eq:mp2}, and the fact that, by Definition~\ref{def:phi-coord}, it holds
 \begin{equation}
   z_0\le \delta(p)^2 \max_{r\in(-2\pi,2\pi)}\frac{r-\sin r}{2r^2} =\frac{\delta(p)^2}{2\pi}. \qedhere
 \end{equation}
 \end{proof}

\section{An alternative proof of the Hardy inequality  \eqref{eq:garofalo}}\label{sec:garofalo}

In this section, we start by proposing a fix for the argument of \cite{Yang2013} in Lemma~\ref{l:yang}, and then we show that this yields a different proof of the classical Hardy inequality \eqref{eq:garofalo} by Garofalo and Lanconelli in Proposition~\ref{p:garofalo}. 

To this aim, we define the vector field $T\in \Gamma(\H^n\setminus\cZ)$ by
\begin{equation}\label{eq:T}
T = \nH\delta -{(\Phi_*w)^ {-1}}\Xi.
\end{equation}
Here, $\Xi$ is the polar vector field of Definition~\ref{d:polar} and $w$ is the function defined in Proposition~\ref{p:ortho}. 
Then a correct version of \cite[Lemma~3.1]{Yang2013} is the following.

\begin{lem}\label{l:yang}
	Let $0<R_1<R_2$ and $f\in C^1((B_{R_2}\setminus B_{R_1})\setminus\cZ)$. Then,
	\begin{equation}
		\int_{\partial B_1} f(\varrho_{R_2}p)\,d\sigma(p) - \int_{\partial B_1} f(\varrho_{R_1}p)\,d\sigma(p) = \int_{B_{R_2}\setminus B_{R_1}} \langle \nH f, T\rangle\frac{1}{\delta^{2n+1}}\,dp.
	\end{equation}
	Here, we let $B_t=\{\delta<t\}$ and denoted by $d\sigma$ the surface measure of $\partial B_1$.
\end{lem}

\begin{proof}
By Proposition~\ref{p:ortho}, we have that $\Phi^*T=(1,\frac{rv(r)}{tw(r)}\tilde J\varpi,0)$. Then, by Lemma~\ref{prop:null_div} we have
\begin{equation}
\int_{B_{R_2}\setminus B_{R_1}} \frac{\langle \nH f, T\rangle}{\delta^{2n+1}}\,dp 
= \int_{-2\pi}^{2\pi} \int_{\S^{2n-1}} \left(\int_{R_1}^{R_2} \partial_t(\Phi^*f)\,dt\right)\,d\varpi\,\mu(r)dr.
\end{equation}
Finally, an integration by parts yields
\begin{equation}
	\begin{split}
		 \int_{B_{R_2}\setminus B_{R_1}} \frac{\langle \nH f, T\rangle}{\delta^{2n+1}}\,dp 
	 &= \int_{-2\pi}^{2\pi} \int_{\S^{2n-1}} \left((\Phi^*f)|_{t=R_2}-(\Phi^*f)|_{t=R_1}\right)\,d\varpi\,\mu(r)dr\\
	 &=\int_{\partial B_1} f(\varrho_{R_2}p)\,d\sigma(p) - \int_{\partial B_1} f(\varrho_{R_1}p)\,d\sigma(p) .\qedhere
	\end{split}
\end{equation}
\end{proof}

\begin{prop}\label{p:garofalo}
	For any $u\in C^\infty_c(\H^n\setminus\{0\})$ it holds
	\begin{equation}
		 \int_{\bH^n}|\nH u|^2\,dp \ge n^2 \int_{\bH^n}\frac{|u|^2}{N^2}|\nH N|^2\,dp.
	\end{equation}
	Here, $N(\xi,z):=\sqrt[4]{|\xi|^4+16z^2}$ is the Koranyi gauge.
\end{prop}

\begin{proof}
	Applying Lemma~\ref{l:yang} to $f=\left({u\delta^n}{|T|^{-1}}\right)^2$ and letting $R_1\downarrow 0$ and $R_2\uparrow +\infty$, since $f$ has compact support outside of the origin, we obtain
	\begin{equation}
		n\int_{\H^n}\frac{u^2}{\delta^2} |T|^{-2}\,dp = -\int_{\H^n}\frac{u}{\delta}|T|^{-2}\langle \nH u, T\rangle.
	\end{equation}
	Then, as in the proof of Theorem~\ref{thm:hardy-intro2}, by Cauchy-Schwarz inequality we obtain
	\begin{equation}
		n^2\int_{\H^n}\frac{u^2}{\delta^2} |T|^{-2}\,dp\le \int_{\H^n}|\nH u|^2\,dp.
	\end{equation}
	Here, we used that, again by Cauchy-Schwarz inequality $|\langle \nH u, T\rangle|\le|\nH u||T|$.
	
	In order to complete the proof, we are left to show that $|T|^2\delta^2=N^2/|\nH N|^2$. By definition, we have 
	\begin{equation}
		\Phi^*|T|^2 = \frac{1+w^2}{w^2}.
	\end{equation}
	On the other hand, following the computations in the proof of Proposition~\ref{p:koranyi}, we have
	\begin{equation}
		\Phi^*\left(\frac{N^2}{|\nH N|^2}\right) = 2\frac{t^2}{r^2}\frac{r^2-2\cos r-2r\sin r+2}{1-\cos r}.
	\end{equation}
	The conclusion then follows at once by direct computations.
\end{proof}

\appendix
\section{An Euclidean non-radial Hardy inequality}\label{a:euclidean}

In this section we present an Euclidean version of Theorem~\ref{thm:hardy-intro2}.
Let $x=(x',x_d)\in \bR^{d-1}\times\bR$.
We consider coordinates $(t,\varpi,\varphi)\in \R_+\times  \bS^{d-2}\times (-\pi/2,\pi/2)$ in $\bR^d$, defined by
\begin{equation}
	(x', x_d) = t(\varpi\cos\varphi, \sin\varphi).
\end{equation}
In this case, the polar vector field is $\Xi=\frac1t\partial_\varphi$, which is unit thanks to the fact that
\begin{equation}\label{eq:tan}
	\varphi = \arctan\frac{x_d}{\|x'\|}.
\end{equation}
Moreover, the volume form becomes $t^{d-1}\cos^{d-2}\varphi \,dt\,d\varphi\,d\sigma(\varpi)$, where $d\sigma$ is the standard volume on $\S^{d-2}$.

Let us consider a spherical cap $\Sigma\subset\bS_+^{d-1}$, and let $C_\Sigma$ be the associated Euclidean cone. We can always assume it to be centered on the $d$-th coordinate axis, i.e., $C_\Sigma=\{\varphi>a_\Sigma\}$ for some $a_\Sigma\in(0,\pi/2)$.
We have the following.

\begin{thm}
Let $d\ge 3$. Then, letting $\psi(x',x_d)=x_d/\|x'\|$, we have
\begin{equation}
	\int_{C_\Sigma} \frac{|\langle\nabla u, \Xi\rangle|^2}{\psi}\,dx \ge \left(\frac{d-2}{2}\right)^2 \int_{C_\Sigma} \frac{|u|^2}{|x|^2}\psi\,dx,
	\qquad
	\forall u \in C^\infty_c(C_\Sigma).
\end{equation}
Moreover, the inequality is sharp.
\end{thm}

\begin{proof}
	Let $v\in C^\infty_c(C_\Sigma)$. An integration by part yields
	\begin{equation}
		\begin{split}
			\int_{C_\Sigma} \langle\nabla v, \Xi\rangle\,dx
			&= \int_0^{+\infty} t^{d-1}\,dt \int_{\bS^{d-2}}d\sigma(\varpi)\int_{a_\Sigma}^{\pi/2} \frac{1}{t}\partial_\varphi v \cos^{d-2}\varphi\,d\varphi \\
			&= (d-2)\int_0^{+\infty} t^{d-1}\,dt \int_{\bS^{d-2}}d\sigma(\varpi)\int_{a_\Sigma}^{\pi/2} \frac{v}{t} \cos^{d-3}\varphi\sin\varphi\,d\varphi \\
			&= (d-2)\int_{C_\Sigma} \frac{v}{|x|}\psi\,dx
		\end{split}
	\end{equation}
	Here, we used that $\psi=\tan\varphi$ in polar coordinates, thanks to \eqref{eq:tan}.
	The desired inequality then follows from the above, choosing $v = u^2/t$ and applying Cauchy-Schwarz on the l.h.s..
	
	To obtain the sharpness, observe that choosing $u(t,\varphi) = \eta(t)(\cos\varphi)^\gamma$ where $\eta(t)$ is any cutoff with compact support and $\gamma\in\R$, we have
	\begin{equation}
		\int_{C_\Sigma} \frac{|\langle\nabla u, \Xi\rangle|^2}{\psi}\,dx = \gamma^2 \int_{C_\Sigma} \frac{|u|^2}{|x|^2}\psi\,dx.
	\end{equation}
	Direct computations show that the last integral is finite if and only if $\gamma>\frac{2-d}{2}$. Then, the statement follows via a cut-off argument as the one in the proof of Theorem~\ref{thm:hardy-intro2}, letting $\gamma\to\left(\frac{2-d}{2}\right)^+$.
\end{proof}

Since the function $\psi$ is unbounded on the vertical axis $\{x' =0\}$, the above does not yield any upper bound on $c^ {\perp,\text{eucl}}_n(\Sigma)$. Indeed, we can only recover the following lower bound, which is not sharp but correctly shows that $c^ {\perp,\text{eucl}}_n(\Sigma)\to +\infty$ as $\Sigma$ degenerates to a point (i.e., $a_\Sigma\uparrow \pi/2$).

\begin{cor}
	Let $\Sigma = \{\varphi>a_\Sigma\}\subset\S^{d-1}$ be a spherical cap. Then, 
	\begin{equation}
		c^ {\perp,\text{eucl}}_n(\Sigma)\ge \left(\frac{d-2}{2}\right)^2 \tan^2 a_\Sigma.
	\end{equation}
\end{cor}

\bibliographystyle{abbrv}
\bibliography{biblio}

\end{document}